\documentclass[12pt,twoside]{article}
\usepackage[T1]{fontenc}
\usepackage{textcomp}
\usepackage{mathrsfs}
\usepackage{amsmath}
\usepackage{amssymb}
\usepackage{fancyhdr}
\usepackage{latexsym}
\usepackage{bbding}
\usepackage{mathrsfs}
\usepackage{wasysym}
\usepackage{cite}
\usepackage{hyperref}
\usepackage{multicol,graphics}
\usepackage{xcolor}

\RequirePackage{times}

\def\XXint#1#2#3{{\setbox0=\hbox{$#1{#2#3}{\int}$ }
		\vcenter{\hbox{$#2#3$ }}\kern-.6\wd0}}

\newtheorem{theorem}{Theorem}[section]
\newtheorem{lemma}[theorem]{Lemma}

\newtheorem{definition}[theorem]{Definition}
\newtheorem{remark}[theorem]{Remark}

\numberwithin{equation}{section}
\newenvironment{proof}[1][Proof]{\noindent\textbf{#1.} }{\hfill $\Box$}

\allowdisplaybreaks[4]
\makeatletter
\begin{document}
	
	\title{Energy equality of the weak solutions to  non-Newtonian fluids equations\footnote{Authors are listed alphabetically by surname then given name. Authors equally share the first authorship. }}
	\author{ Yi Feng\footnote{mx120220310@stu.yzu.edu.cn}, Weihua Wang\footnote{Corresponding author: wangvh@163.com, wangweihua15@mails.ucas.ac.cn (W.Wang).}\\
		[0.2cm] {\small School of Mathematical Science, Yangzhou University,}\\
		[0.2cm] {\small  Yangzhou, Jiangsu 225009,  China}}
	\date{\today}
	\maketitle
	\begin{abstract}
		In this paper, we study the problem of energy equality for weak solutions of the 3D incompressible non-Newtonian fluid equations with initial  value conditions. We derive new sufficient conditions via Sobolev multiplier spaces that guarantee the validity of the energy equality. Moreover, the aforementioned equations are often associated with the uniqueness problem of weak solutions for non-Newtonian fluids, which, in a certain sense, constitutes the positive counterpart of Onsager's conclusion for non-Newtonian fluids.
	\end{abstract}
	\smallbreak
	
	\smallskip
	
	
	\section{Introduction}
We consider the following 3D incompressible non-Newtonian  fluid equations  with initial value in the appropriate space
\begin{equation}\label{eq1.1}
	\begin{cases}
		u_{t} +(u\cdot \nabla)u - \mu\text{div}\left(|D(u)|^{r-2}D(u)\right) + \nabla p=0,\quad &(x,t)\in \mathbb{R}^{3}\times (0,T),\\
		\text{div}~ u =0, \quad &(x,t)\in \mathbb{R}^{3}\times (0,T),\\
		u(x,0)=u_{0}, \quad &x\in \mathbb{R}^{3}.
	\end{cases}
\end{equation}
where
\begin{equation*}
	D(u)=\frac{1}{2}\left(\nabla u+(\nabla u)^{T}\right)
\end{equation*}
is the symmetric part of the velocity gradient, $u$ and $p$ represent the velocity  field and  scalar pressure, respectively. $u_0$ denotes  the  initial value of the velocity field, which also  satisfies  $\text{div}~u_0=0$ in distribution sense.  To simplify the problem, we take the viscosity coefficient as unity and set the external force to vanish. The mathematical theory concerning power-law fluids and their generalization (i.e. non-Newtonian  fluids) is  quite rich, and we  may refer to \cite{CrFG2026,Galdi08,BuKP2019,BuMM2023} and  the references  therein.

 Here we focus on the positive part of Onsager's conjecture \cite{Onsager1949} for non-Newtonian fluids, i.e., the conditions under which the energy equality holds. The energy equality for non-Newtonian fluid equations is expressed as follows:
\begin{equation}\label{eq1.2}
	\frac{1}{2}\|u(t)\|^{2}_{L^{2}} + \int_{0}^{t}\|D(u)(\tau)\|^{r}_{L^{r}}{\rm d}\tau = \frac{1}{2}\|u_0\|^{2}_{L^{2}},
\end{equation}
for any $~0\le t<T$.  Our aim is to show the energy equality for a larger class of solutions, which is  closely related to the  uniqueness of weak solution  of  non-Newtonian  fluids \cite{BuKP2019,BuMM2023}.

 In the case of $r=2$,  Eq. \eqref{eq1.1} describes  Newtonian fluid (Navier-Stokes equations), that is
\begin{equation}\label{eq1.3}
	\begin{cases}
		u_{t} +(u\cdot \nabla)u - \mu\triangle u + \nabla p=0,\quad &(x,t)\in X\times (0,T),\\
		\text{div} ~u =0, \quad &(x,t)\in X\times (0,T),\\
		u(x,0)=u_{0}, \quad &x\in X.
	\end{cases}
\end{equation}
with the appropriate boundary conditions, where $X$ denotes $\mathbb{R}^{3}$ or domain $\Omega$ in $\mathbb{R}^{3}$.
Correspondingly, the energy equality of Newtonian fluid equations is
\begin{equation}\label{eq1.4}
	\frac{1}{2}\|u(t)\|^{2}_{L^{2}} + \int_{0}^{t}\|\nabla u(\tau)\|^{2}_{L^{2}}{\rm d}\tau = \frac{1}{2}\|u_0\|^{2}_{L^{2}},
\end{equation}
for any $~0\le t<T$.

For the Leray-Hopf weak solutions of  the Navier-Stokes  equations, Lions \cite{LJL60} obtained the energy equality in $L^{4}(0,T; L^{4}(\Omega))$. Some years later, Serrin \cite{SJ63}  derived the energy conservation criterion: $L^{p}(0,T; L^{q}(\Omega))$ with $\frac{2}{p} + \frac{n}{q} \leq 1$(scaling invariant space).  With Serrin's results, Masuda \cite{MK84} obtained the uniqueness of the weak solution of the Navier-Stokes  equations  \eqref{eq1.3}.
 And Shinbrot \cite{SM74} improved Serrin's result to $\frac{2}{p} + \frac{2}{q} \leq 1$ for $q\geq 4$, which does not depend on the spatial dimension $n$ and  is larger than that of the scaling invariant space given by Serrin \cite{SJ63}.

  Cheskidov et al.~\cite{CheCoFrSh08} acquired the energy equality for  the  Euler equations in the largest   function space  $L^{3}(0,T; B^{\frac{1}{3}}_{3,\infty}(\mathbb{R}^{3}))$. From this result and the  interpolation  with $L^{\infty}_{t}L^{2}_{x}\cap L^{2}_{t}H^{1}_{x}$, Cheskidov \& Luo \cite{ChesLuo20} gained the energy equality in the
following class:\newline
$L^{p}(0,T; B^{\alpha}_{q,\infty}(\mathbb{R}^{3}))$  for\\
 $R(p,q)-1-\alpha =
  \begin{cases}
    1+\alpha-\frac{3}{q} \text{~with~} 3\leq p \text{~and~} 1\leq q\leq p\leq\infty;\\
    \frac{1}{5} - \frac{\alpha}{5} + \frac{3}{5q}  \text{~with~} 1\leq\frac{1}{p}+\frac{2}{q}, 0<p\leq 3\text{~and~} 1\leq q\leq\infty;\\
    \frac{1}{q} \text{~with~} \frac{1}{p}+\frac{2}{q}\leq 1 \text{~and~} 1\leq q\leq p\leq\infty
  \end{cases}$\\
or
\begin{equation*}
   L^{p}_{w}(0,T; B^{\alpha}_{q,\infty}(\mathbb{R}^{3})) \text{~for~} R(p,q)-1-\alpha=\frac{1}{q} \text{~with~} \frac{1}{p}+\frac{2}{q}\leq 1 \text{~and~} 1\leq q < p\leq\infty,
\end{equation*}
where $B^{\alpha}_{q,\infty}$  presents the inhomogeneous Besov space and $R(p,q)\triangleq \frac{2}{p}+\frac{3}{q}$.
In the same year, Berselli and Chiodaroli \cite{BC20} showed the energy equality under the following conditions
\begin{equation*}
  L^{p}(0,T; W^{1,q}(\Omega)) \text{~for~} R(p,q)-2=
  \begin{cases}
    2-\frac{3}{q} \text{~with~} \frac{3}{2}<q<\frac{9}{5};\\
    \frac{3}{5q}  \text{~with~} \frac{9}{5}\leq q \leq 3;\\
    \frac{3}{q} - \frac{4}{q+2} \text{~with~} 3<q\leq\infty.
  \end{cases}
\end{equation*}

 Galdi \cite{GGP19} proposed when
  \begin{equation*}
  	u\in L^{4}(0,T;L^{4}(\Omega)),
  \end{equation*}
   by  virtue of the  mollifying procedure and duality argument, the energy equality of the distributional solutions to the Navier-Stokes equations still holds true, what is impressive is that the Leray-Hopf class is not necessary.
   Moreover, Berselli and Chiodaroli \cite{BC20} also presented that if
   \begin{equation*}
   	u\in L^{p}(0,T;L^{q}(\Omega))\quad\text {with}\quad \frac{2}{p}+\frac{2}{q}= 1; ~q\geq4,
   \end{equation*}
   the energy equality of the distributional solutions to \eqref{eq1.3} holds. This result follows from a duality argument, analogous to the approach adopted  in Galdi \cite{GGP19} and covers the full range of exponents. In  the work of  Wu \cite{Wu24}, the problem of energy equality for the distributional solutions to the fractional Navier-Stokes equations is considered: if $u\in L^{\frac{4\alpha}{2\alpha-1}}(0,T;L^{4}(\mathbb{R}^{3}))$ with $\alpha\geq 1$, the energy equality for the distributional solutions holds true.

  In the case of $r\neq2$,  Eq. \eqref{eq1.1} determines the motion behavior of incompressible homogeneous non-Newtonian fluid.
   Yang \cite{YJ19} extended Shinbrot's energy equality \cite{SM74} for the Navier-Stokes equations to non-Newtonian fluid in
   \begin{equation*}
     L^{p}(0,T; L^{q}(\mathbb{R}^{n})), \begin{cases}
       \frac{1}{p} + \frac{1}{q} \leq 1-\frac{1}{r}, \frac{2r}{r-1}\leq q \leq  \frac{2r}{r-2}, r>2;\\
       \frac{1}{p} + \frac{r-1}{q} \leq \frac{r-1}{2}, q\geq \frac{2r}{r-1}, r\leq2.
     \end{cases}
   \end{equation*}
   Shortly thereafter, Zhang\cite{ZZ19} generalized Yang's results \cite{YJ19} that if $u$ satisfies one of the following conditions:
   \begin{equation*}
   	\begin{cases}
   		(\romannumeral 1) &u\in L^{p}(0,T;L^{q}(\mathbb{R}^{n})),~ \frac{1}{p}+\frac{r-1}{q}= \frac{r-1}{2},~\frac{2r}{r-1}\leq q<\infty,~r>\frac{2(n+1)}{n+2};\\
   		(\romannumeral 2)&u\in L^{3}(0,T;\dot{B}^{\frac{1}{2}}_{\frac{6n}{2n+1},2}(\mathbb{R}^{n})),~r>\frac{2(n+1)}{n+2},
   	\end{cases}
   \end{equation*}
   then the energy equality \eqref{eq1.2} holds true.
   Beir\~{a}o and Yang successfully \cite{BY19} extended Berselli and Chiodaroli's conclusion \cite{BC20} to the case of the non-Newtonian fluids provided the gradient of velocity $\nabla u$ satisfies the following conditions:
   \begin{equation*}
   	\begin{cases}
   		(\romannumeral 1)_1   & \frac{9}{5}<r\leq2,~\frac{9-3r}{2}<q\leq \frac{9}{5},~p=\frac{q(5r-9)}{3r+2q-9},  ~\nabla u \in L^{\frac{q(5r-9)}{3r+2q-9}}(0,T;L^{q}(\Omega));\\
   		(\romannumeral 1)_2   & 2<r<\frac{11}{5},~\frac{3r}{5r-6}\leq q\leq \frac{9}{5},~p=\frac{q(5r-9)}{3r+2q-9},  ~\nabla u \in L^{\frac{q(5r-9)}{3r+2q-9}}(0,T;L^{q}(\Omega));\\
   		(\romannumeral 2)_1 & r<\frac{11}{5},~q>\frac{9}{5},~p=\frac{5q}{5q-6}, ~\nabla u \in L^{\frac{5q}{5q-9}}(0,T;L^{q}(\Omega));\\
   		(\romannumeral 2)_2 & r\geq\frac{11}{5}.
   	\end{cases}
   \end{equation*}
 For results regarding the energy equality  for non-Newtonian fluids, we also refer to \cite{SinBa23,WangMH22} and   the references   therein.

   Chen and Zhang \cite{CZ23} started with   \cite[Theorem 1]{BC20} and  identified several sufficient conditions involving  the Sobolev multiplier space;  the results are as follows:
\begin{equation*}
	\begin{cases}
		(\romannumeral 1)&|u|^{\frac{2}{\alpha+\theta}-2}u\in L^{\frac{2(\alpha +\theta)}{1-\theta}}\mathcal{M}^{3}(\dot{H}^{\frac{\theta}{\alpha+\theta}}\to L^{2(\alpha+\theta)}), ~ \alpha,\theta\ge0,~\frac{1}{2}\le\alpha+\theta\le1;\\
		(\romannumeral 2)&|u|^{\frac{10\theta^{2}-23\theta+7}{3+8\theta-5\theta^{2}}}u\in L^{\frac{2(-5\theta^{2}+8\theta+3)}{5(1-\theta)^{2}}}\mathcal{M}^{3}(\dot{H}^{1}\to L^{\frac{-5\theta^{2}+8\theta+3}{3-2\theta}}),\\
		&|\nabla u|^{\frac{10-5\theta}{7-3\theta}} \in L^{\frac{(2\alpha+2)(7-3\theta)}{5-5\theta+12\alpha-8\alpha\theta}}\mathcal{M}^{3}(\dot{H}^{\frac{1}{\alpha+1}}\to L^{1}),~\alpha\geq0,~0\leq\theta\leq\frac{1}{2};\\
		(\romannumeral 3)&\nabla u \in L^{\frac{(4-2\theta)(1+\alpha)}{1-\theta+(3-\theta^{2})\alpha}}\mathcal{M}^{3}(\dot{W}^{\frac{1+\theta}{1+\alpha},\frac{6-3\theta}{3-\theta^{2}}}\to L^{\frac{12-6\theta}{9-5\theta}}),~
		u \in L^{\frac{4-2\theta}{1-\theta}}(L^{\frac{12-6\theta}{3-\theta}}),\\
		&~\alpha\geq 0,~-1\leq \theta <0.
	\end{cases}
\end{equation*}
In a recent extension of result ($i$)  by Chen and Zhang \cite{CZ23} on the Navier-Stokes equations, Feng and Wang \cite{FW26,FW25}  obtained  its applicability to the Navier-Stokes-Maxwell equations, the inhomogeneous equations of hydrodynamics and magnetohydrodynamics.

Inspired by \cite{BC20} and \cite{CZ23},  our  primary  objective is to find weaker sufficient conditions   that ensure  the energy equality of the weak solutions to   non-Newtonian fluid equations still holds true (i.e., \eqref{eq1.2}) in Sobolev multiplier spaces,  we use inequalities about  multiplier  spaces  come to some conditions, involving $|u|^{\frac{2}{\alpha+\theta}-2}u$ in the problem of the energy equality for the weak solutions
to  non-Newtonian  fluid  equations. The main purpose  in  this paper is to generalize  the results of Chen and Zhang \cite{CZ23}  to
non-Newtonian fluids case, by mollifying approximation, we can make sure that each term is bounded in the sense of integral;  especially $\text{div}(|D(u)|^{r-2}D(u))$ brings difficulties.  Specifically, our main results are as follows:
\begin{theorem}\label{th1.1}
		Let $u$ be a weak solution of \eqref{eq1.1}   on some time interval $[0,T]$ with $u_0\in L_{\sigma}^{2}(\mathbb{R}^{3})$ and$~0<T\leq\infty$. If $~u$ satisfies $|u|^{\frac{2}{\alpha+\theta}-2}u\in L^{\frac{r(5r-6)(\alpha +\theta)}{5r^{2}-(11+2\theta)r+6}}\mathcal{M}^{3}(\dot{H}^{\frac{\theta}{\alpha+\theta}}\to L^{\frac{r(\alpha+\theta)}{r-1}})$, where $\alpha,\theta\ge0$, $\frac{1}{2}\le\alpha+\theta\le1$, $r\geq
		2$
		then the energy equality \eqref{eq1.2} holds for any $0\le t<T$.
\end{theorem}

\begin{remark}
  In Beir\~{a}o and Yang \cite{BY19}, we can observe that $p=\frac{9(5r-9)}{3r+2q-9}\geq3$ in Theorem 4.3 $(i)_2$, correspondingly, taking $\alpha=\frac{1}{2}$ and $\theta=0$ in Theorem \ref{th1.1}, we have $ \frac{r(5r-6)(\alpha +\theta)}{5r^{2}-(11+2\theta)r+6}\geq \frac{r}{2(r-1)}>\frac{1}{2}$, additionally, taking $\alpha=1$ and $\theta=0$ in Theorem ~\ref{th1.1}, we have $u~\in L^{\frac{r}{r-1}}\mathcal{M}^{3}(\dot{H}^{\frac{\theta}{\alpha+\theta}}\to L^{\frac{r(\alpha+\theta)}{r-1}})$, where $\frac{r}{r-1}>1$. Obviously, the indicators of time integrability have a wider range than that in \cite{BY19}.
\end{remark}

	\begin{theorem}\label{th1.2}
		Let $u$ be a weak solution of \eqref{eq1.1}   on some time interval $[0,T]$ with $u_0\in L_{\sigma}^{2}(\Omega)$ and$~0<T\leq\infty$. If $~u$ satisfies 	\begin{equation*}
			|u|^{\frac{10\theta^{2}-23\theta+7}{3+8\theta-5\theta^{2}}}u\in L^{\frac{(5r-6)(-5\theta^{2}+8\theta+3)}{(5r-6)(-\theta+4)+2(5\theta^{2}-8\theta-3)}}\mathcal{M}^{3}(\dot{H}^{1}\to L^{\frac{-5\theta^{2}+8\theta+3}{3-2\theta}})
		\end{equation*} and
		\begin{equation*}
			|\nabla u|^{\frac{10-5\theta}{7-3\theta}} \in L^{\frac{(\alpha+1)(7-3\theta)(5r-6)}{2(3-2\theta)(\alpha+1)(5r-6)-2(7-3\theta)}}\mathcal{M}^{3}(\dot{H}^{\frac{1}{\alpha+1}}\to L^{1}),
		\end{equation*}
		where $ \alpha\geq0,~r\geq2,~0\leq\theta<\frac{1}{2}$, the energy equality \eqref{eq1.2} holds for any $0\le t<T$.
	\end{theorem}
\begin{theorem}\label{th1.3}
	Let $u$ be a weak solution of \eqref{eq1.1}   on some time interval $[0,T]$ with $u_0\in L_{\sigma}^{2}(\Omega)$ and$~0<T\leq\infty$. If $~u$ satisfies 	
	\begin{equation*}
		\nabla u \in  L^{p}\mathcal{M}^{3}(\dot{W}^{\frac{1+\theta}{1+\alpha},\frac{6-3\theta}{3-\theta^{2}}}\to L^{\frac{12-6\theta}{9-5\theta}})
	\end{equation*} and
	\begin{equation*}
		u \in L^{q}(L^{\frac{12-6\theta}{3-\theta}}),~ \frac{1}{p}+\frac{1}{q}+\frac{1}{r}=1,~2\leq p<\infty,
	\end{equation*}
	where $\alpha\geq0,~r\ge2\text{ and }-1\leq\theta<0$, the energy equality \eqref{eq1.2} holds for any $0\le t<T$.
\end{theorem}

	The rest of this paper is organized as follows: In Section \ref{Sec2}, we outline some basic definitions and facts, as well as give some key lemmas. And in Section \ref{Sec3} and Section \ref{Sec4}, we give the proof of main results.   Section \ref{Sec5} is an appendix, where we give the proof of three lemmas in Section \ref{Sec3} and Section \ref{Sec4}.
	
\noindent {\bf Notation}.  ~~Throughout the paper,  $A \lesssim B$ denotes $|A|\le C|B|$ with some positive constant $C$. If $C$ depends on a parameter $\varepsilon$,  that is, $C=C(\varepsilon)$, ~$\lesssim$ is replaced by $ \lesssim_{\varepsilon}$,  and  $C(\varepsilon,t_0)$ denotes a constant depending on $\varepsilon$ and $t_0$. 	

\section{ Preliminaries}\label{Sec2}
In this section, we recall the classical Lebesgue  spaces  and Sobolev  spaces. We  then  introduce the multiplier spaces,  along with the definitions  and properties of the solutions  that will  be used in this paper. Finally, we review the Gagliardo-Nirenberg inequality and convolution, which  are frequently used  throughout  this paper.\\
The Lebesgue space $L^{p}(\mathbb{R}^{3})$ consists of all strongly measurable functions $u: \Omega\to \mathbb{R}^3$ with the norm
	\begin{equation*}
		\|u\|_{L^{p}(\mathbb{R}^{3})}:=\begin{cases}
  \left(\int_{\mathbb{R}^{3}}|u(x)|^{p}dx\right)^{\frac{1}{p}}, &1\leq p<\infty;\\
  \underset{\mathbb{R}^{3}}{\text{ess}\sup}|u(x)|, & p=+\infty.
\end{cases}
	\end{equation*}
Sobolev spaces $W^{k,p}(\mathbb{R}^{3})$ is defined as the all functions  satisfying  $$\sum_{|\alpha|\leq k}\|D^{\alpha}u\|_{L^{p}}\leq +\infty$$ equipped with norm $\|u\|_{W^{k,p}}\triangleq \sum_{|\alpha|\leq k}\|D^{\alpha}u\|_{L^{p}}$.  In particular, $W^{k,p}(\mathbb{R}^{3})$ is simply $H^{k}$ when $p=2$.
For convenience, we  denote  the norm of $L^{p}(\mathbb{R}^{3})$  by  $\|\cdot\|_{L^{q}}$ and the norm of $L^{q}(0,T;L^{p}(\mathbb{R}^{3}))$ simply by $\|\cdot\|_{L^{q}(L^{p})}$.

The space $L^{q}(0,T;~L^{p}(\mathbb{R}^{3}))$ consists of all strongly measurable functions $u: [0,T]\to  L^{p}(\mathbb{R}^{3}) $ with	
\begin{equation*}
		\|u\|_{L^{q}(0,T;L^{p}(\mathbb{R}^{3}))}
:=\begin{cases}
  \left(\int_{0}^{T}\|u(t)\|_{L^{p}(\mathbb{R}^{3})}^{q}dt\right)^{\frac{1}{q}}, & 1\leq q<\infty;\\
   \underset{0\leq t\leq T}{\text{ess}\sup}\|u(t)\|_{L^{p}(\mathbb{R}^{3})}, & q=+\infty.
  \end{cases}
\end{equation*}	

In order to give an accurate definition of the weak solution, we also need the following  notations  and function  spaces. $C^{\infty}_{0} (\mathbb{R}^{3})$ is the space of smooth functions with compact support, and we have
\begin{equation*}
		C^{\infty}_{0,\sigma} (\mathbb{R}^{3})=\{u\in C^{\infty}_0 (\mathbb{R}^{3}),\text{div} ~u=0\},
\end{equation*}
then we denote the completion of $C^{\infty}_{0,\sigma} (\mathbb{R}^{3})$ in $L^{2}(\mathbb{R}^{3})$ by $L^{2}_{\sigma}(\mathbb{R}^{3})$ and the completion in $W^{1,p}(\mathbb{R}^{3})$ by $W^{1,p}_{\sigma}(\mathbb{R}^{3})$.

\begin{definition}[Leray-Hopf weak solutions, Definition 1,~\cite{YJ19}]
		
		 Let $r\geq\frac{6}{5}$,  we say that  $u\in L^{\infty}(0,T;L^{2}_{\sigma}(\mathbb{R}^{3}))$ $\cap$  ~$L^{r}(0,T;W^{1,r}_{\sigma}(\mathbb{R}^{3}))$ is a Leray-Hopf weak solution to \eqref{eq1.1} satisfying the initial condition  $u_0\in L_\sigma^{2}(\mathbb{R}^{3})$, if the following assumptions hold:
		
		(i) $u$ is a solution of \eqref{eq1.1} in the sense of distribution
		\begin{equation*}
			\int_{0}^{T}(u,\psi_\tau)-\left(|D(u)|^{r-2}D(u),D(\psi)\right)-((u\cdot\nabla)u,\psi){\rm d}\tau=-(u_0,\psi(0)),
		\end{equation*}
		for all $\psi \in C^{\infty}_0([0,T);C^{\infty}_{0,\sigma}(\mathbb{R}^{3}))$;
		
		(ii) $u$ satisfies the energy inequality
		\begin{equation*}
			\|u(t)\|^{2}_{L^{2}}+2\int_{0}^{t}\|D (u)(\tau)\|^{r}_{L^{r}}{\rm d}\tau\leq \|u_0\|^{2}_{L^{2}},\quad t \in (0,T);
		\end{equation*}
		
		(iii) $\|u(t)-u_0\|^{2}_{L^{2}}\to 0 \text{ as } t\to 0^{+}$.
		
	\end{definition}

$r\geq\frac{2n}{n+2}$ guarantees the boundedness of $\int_{0}^{T} ((u\cdot\nabla)u,\psi){\rm d}\tau$ in  $n$-dimensional space, here we only consider the case of three dimension.
	
	\begin{lemma}[Lemma 2,~\cite{YJ19}]\label{le2.6}
	Let $r>\frac{8}{5}$, $u_0 \in L^{2}_\sigma(\mathbb{R}^{3})$, and $u$ be a weak solution of \eqref{eq1.1}. Then, after suitable redefinition of $u$ on a set of values of $t$ of one-dimensional measure zero, we have
		\begin{equation}\label{eq2.1}
	(u(t),\phi(t))	=(u_0,\phi(x,0))+\int_{0}^{t}\left[( u,\phi_\tau)-((u\cdot\nabla)u,\phi)-\left(\left|D(u)\right|^{r-2}D(u),D(\phi)\right)\right]{\rm d}\tau,
		\end{equation}
	\begin{equation*}
		u\in L^{\infty}(0,T;L^{2}_\sigma)\cap L^{r}(0,T;{W}^{1,r}_{\sigma})
	\end{equation*}
	for any $t \in [0,T)\;$and $\phi \in C^{\infty}_0(0,T;C^{\infty}_{0,\sigma}(\mathbb{R}^{3})).$
	\end{lemma}
  \begin{remark}
    Wolf \cite{WJ07} proved the existence of weak solutions to non-Newtonian  fluid equations \eqref{eq1.1}  in $n$-dimensional space for $r>\frac{2(n+1)}{n+2}$. Hence, the assumption $r>\frac{8}{5}$ is necessary in $\mathbb{R}^{3}$.
  \end{remark}

\begin{definition}[\cite{MS09}]\label{de2.1}
    Let $E$ and $F$ be two Banach spaces embedded in $\mathcal{S}'(\mathbb{R}^{3})$. The multiplier space $\mathcal{M}(E\to F)$ is the set of  all  functions $\varphi$ such that $\varphi u\in F$ for every $u\in E$, with the norm defined by
    \begin{equation*}
        \|\varphi\|_{\mathcal{M}(E\to F)}:=\sup_{u\in E\setminus\{0\}}\frac{\|\varphi u\|_F}{\|u\|_E}<\infty.
    \end{equation*}
    Moreover, if $\varphi$ is a  $d$-dimensional  vector-valued function, i.e., $\varphi=(\varphi^{1},\dots,\varphi^{d})\in\mathcal{M}(E\to F)$,  the corresponding norm is given by
    \begin{equation*}
        \|\varphi\|_{\mathcal{M}^{d}(E\to F)}^{2}:=\sum_{1\le i\le d}\|\varphi^{i}\|_{\mathcal{M}(E\to F)}^{2}.
    \end{equation*}
    It is easy to observe that there exists an equivalence relation between multiplier spaces and Lebesgue spaces. For any $q\in (1,\infty]$, we have $L^{q}(\Omega)=(L^{p}(\Omega))'$ where $\Omega\subset\mathbb{R}^{3}$ is a domain, i.e.
    \begin{equation*}
        L^{q}(\Omega)=\mathcal{M}(L^{p}(\Omega)\to L^{1}(\Omega))
    \end{equation*}
    with $\frac{1}{p}+\frac{1}{q}=1$.
\end{definition}
		
\begin{remark}[\cite{CZ23}]\label{re2.2}
		 For $\mathcal{M}^{3}(\dot{H}^{\frac{\theta}{\alpha+\theta}}\to L^{\frac{r(\alpha+\theta)}{r-1}})$, according to Sobolev embedding, we easily get $\dot{H}^{\frac{\theta}{\alpha+\theta}}\hookrightarrow L^{\frac{6(\alpha+\theta)}{3\alpha +\theta }}$, since $\frac{1}{\frac{r(\alpha+\theta)}{r-1}}=\frac{1}{\frac{6(\alpha+\theta)}{3\alpha +\theta }}+\frac{1}{\frac{6r(\alpha+\theta)}{6r-6-3\alpha r-\theta  r}}$, we gain
		\begin{equation*}
			L^{\frac{6r(\alpha+\theta)}{6r-6-3\alpha r-\theta  r}}\hookrightarrow\mathcal{M}^{3}(\dot{H}^{\frac{\theta}{\alpha+\theta}}\to L^{\frac{r(\alpha+\theta)}{r-1}}),
		\end{equation*}
		for any $r\geq2$, which guarantees $\frac{6r(\alpha+\theta)}{6r-6-3\alpha r-\theta  r}>0$.
\end{remark}
In order to derive an initial integral equation by constructing a test function from the weak solution of the equation, we first recall some properties of the standard mollifier.	
Let $\eta$ be a standard mollifier (\S.5, p713, \cite{ELC98}), that is, $\eta(t)=C\exp\left(\frac{1}{|t|^{2}-1}\right)\chi_{\{|t|\leq 1\}}$, where the constant $C\geq 0$ selected to integrate to unity and $\chi_{\{|t|\leq 1\}}$ is the indicator function of $\{|t|\leq 1\}$.  For any $\varepsilon>0$,  we set the  rescaled  mollifier $\eta_{\varepsilon}(t):=\frac{1}{\varepsilon}\eta(\frac{t}{\varepsilon})$, then $\int_{\mathbb{R}}\eta_{\varepsilon}(t){\rm d}t=1$.   For any function $f\in L^{1}_{loc}$, define its mollification as
\begin{equation*}
	f_{\varepsilon}(t)=(\eta_{\varepsilon}*f)(t)=\int_{-\varepsilon}^{\varepsilon}\eta_\varepsilon(t-\sigma)f(\sigma){\rm d}\sigma.
\end{equation*}

\begin{lemma}[Gagliardo-Nirenberg inequality,  ~Lemma 2.2, ~\cite{Wu24} or \cite{Nirenberg11}]\label{le2.3}
		Let $1\leq q,s<\infty$ and $m\leq k$.
		Suppose that $j$ and $\vartheta$ satisfy $m\leq j\leq k$, $0\leq \vartheta\leq 1$ and define $p\in [1,+\infty]$ by
		
		\begin{equation*}
			j- \frac{3}{p}=\left(m-\frac{3}{s}\right)\vartheta+\left(k-\frac{3}{q}\right)(1-\vartheta).
		\end{equation*}
		Then the inequality holds:
		\begin{equation*}
			\|\nabla^{j} u\|_{L^{p}}\leq C	\|\nabla^{m}u\|_{L^{s}}^{\vartheta}\|\nabla^{k} u\|_{L^{q}}^{1-\vartheta}, ~ u\in W^{m,s}(\mathbb{R}^{3})\cap W^{k,q}(\mathbb{R}^{3}),
		\end{equation*}
		where constant $C\geq 0$. Here, when $p=\infty$, we require that $0<\vartheta<1$.
	\end{lemma}
	
	\begin{lemma}\label{le2.4}
      Let $m=1$, $j=\frac{\theta}{\alpha+\theta}$, $k=0$, $p=q=2,~s=r$, and $\vartheta=\frac{2r\theta}{(\alpha+\theta)(5r-6)}$ in Lemma \ref{le2.3}, and according to Remark \ref{re2.2} we get the estimate
		\begin{eqnarray*}
    \||u|^{2}\|_{L^{\frac{r}{r-1}}} &=& \||u|^{\frac{2}{\alpha+\theta}} \|^{\alpha+\theta}_{L^{\frac{r(\alpha+\theta)}{r-1}}}=\||u|^{\frac{2}{\alpha+\theta}-2}u\cdot u\|^{\alpha+\theta}_{L^{\frac{r(\alpha+\theta)}{r-1}}} \\
                               &\leq& \||u|^{\frac{2}{\alpha+\theta}-2}u\|^{\alpha+\theta}_{\mathcal{M}^{3}(\dot{H}^{\frac{\theta}{\alpha+\theta}}\to L^{\frac{r(\alpha+\theta)}{r-1}})}\|u\|^{\alpha+\theta}_{\dot{H}^{\frac{\theta}{\alpha+\theta}}} \\
                                &\leq& \||u|^{\frac{2}{\alpha+\theta}-2}u\|^{\alpha+\theta}_{\mathcal{M}^{3}(\dot{H}^{\frac{\theta}{\alpha+\theta}}\to L^{\frac{r(\alpha+\theta)}{r-1}})}\|\nabla u\|^{\frac{2r\theta}{5r-6}}_{L^{r}}\| u\|^{\frac{\alpha(5r-6)+3\theta(r-2)}{5r-6}}_{L^{2}}
                                \end{eqnarray*}
       with $r\geq2$.
	\end{lemma}	
	
\section{Proof of Theorem \ref{th1.1}} \label{Sec3}
Under the assumptions of Theorem \ref{th1.1}, in order to estimate the advection term in \eqref{eq1.1}, we need the following lemma:
 \begin{lemma}\label{le2.7}
 	For $u \text{~in~} L^{\infty}(0,T;L^{2}_{\sigma}(\mathbb{R}^{3}))\cap L^{r}(0,T;W^{1,r}_{\sigma}(\mathbb{R}^{3}))$ and  $|u|^{\frac{2}{\alpha+\theta}-2}u$   in\\  $L^{\frac{r(5r-6)(\alpha +\theta)}{5r^{2}-(11+2\theta)r+6}}\mathcal{M}^{3}(\dot{H}^{\frac{\theta}{\alpha+\theta}}\to L^{\frac{r(\alpha+\theta)}{r-1}})$, where $\alpha,\theta\ge0$, $\frac{1}{2}\le\alpha+\theta\le1$ and $r\geq2$, by Lemma \ref{le2.4} we have	
 	\begin{eqnarray*}
 		& &\left|\int_{0}^{t_0}((u\cdot\nabla)u,u){\rm d}\sigma \right|\\
 		&\lesssim& \||u|^{\frac{2}{\alpha+\theta}-2}u\|^{\alpha+\theta}_{L^{\frac{r(5r-6)(\alpha +\theta)}{5r^{2}-(11+2\theta)r+6}}\mathcal{M}^{3}(\dot{H}^{\frac{\theta}{\alpha+\theta}}\to L^{\frac{r(\alpha+\theta)}{r-1}})}\|\nabla u\|^{\frac{2r\theta}{5r-6}+1}_{L^{r}(L^{r})}\|u\|^{\frac{\alpha(5r-6)+3\theta(r-2)}{5r-6}}_{L^{\infty}(L^{2})}.
 	\end{eqnarray*}
 \end{lemma}
Since the proof of Lemma \ref{le2.7} relies only on tedious calculations, along with the Gagliardo-Nirenberg inequality and the H"{o}lder inequality, we will present the detailed  proof in the appendix.

\begin{proof}[Proof of Theorem \ref{th1.1}]	
We can choose a sequence of $u^{j}\in C^{\infty}_0(0,T;C^{\infty}_{0,\sigma}(\mathbb{R}^{3}))$ converging to $u$ in $ L^{2}(0,T;L^{2}_\sigma(\mathbb{R}^{3}))\cap L^{r}(0,T;W^{1,r}_{\sigma}(\mathbb{R}^{3}))$, which is based on the fact that $C^{\infty}_{0,\sigma}(\mathbb{R}^{3})~$is dense in $L^{2}_\sigma(\mathbb{R}^{3})\cap W^{1,r}_{\sigma}(\mathbb{R}^{3})$, then we have
\begin{equation*}
	u^{j}_\varepsilon(t)=(\eta_\varepsilon*u^{j})(t)=\int_{0}^{t_0}\eta_\varepsilon(t-\sigma)u^{j}(\sigma){\rm d}\sigma
\end{equation*} for any fixed $t_0<T$.

Taking $u^{j}_\varepsilon\;$as test function in \eqref{eq2.1} and replacing $t\;$with fixed $t_0$, we get
\begin{eqnarray}\label{eq3.1}
	& &\left(u(t_0),u^{j}_\varepsilon(t_0)\right)\nonumber\\
	&=&\int_{0}^{t_0}\eta_\varepsilon(t_0-\sigma)(u(t_0),u^{j}(\sigma)){\rm d}\sigma \nonumber\\
	&=&\int_{0}^{t_0}\eta_\varepsilon(-\sigma)(u_0,u^{j}(\sigma)){\rm d}\sigma \nonumber\\
	& &-\int_{0}^{t_0}\int_{0}^{t_0}\eta_\varepsilon(\tau-\sigma)\left[((u\cdot\nabla)u(\tau),u^{j}(\sigma)) +(|D(u)|^{r-2}D(u),D(u^{j}))\right]{\rm d}\sigma {\rm d}\tau\nonumber\\
	& & +\int_{0}^{t_0}\int_{0}^{t_0}\frac{d}{d\tau}[\eta_\varepsilon(\tau-\sigma)](u(\tau),u^{j}(\sigma)){\rm d}\sigma {\rm d}\tau
\end{eqnarray}
Letting $j$ to infinity in \eqref{eq3.1} for fixed $\varepsilon$,  we next provide a detailed estimate of the boundedness of the linear terms in \eqref{eq3.1}.
\begin{eqnarray*}
	& & 	\left|\int_{0}^{t_0}\eta_\varepsilon(t_0-\sigma)(u(t_0),u^{j}(\sigma)-u(\sigma)){\rm d}\sigma\right| \\
	&\leq~~~& \int_{0}^{t_0}\eta_\varepsilon(t_0-\sigma)\|u(t_0)\|_{L^{2}}\|u^{j}(\sigma)-u(\sigma)\|_{L^{2}}{\rm d}\sigma\quad(\text{ H\"{o}lder's inequality})\\
	&\lesssim_{\varepsilon,t_0}&  \|u\|_{L^{\infty}(L^{2})}\|u^{j}-u\|_{L^{2}(L^{2})}
\end{eqnarray*}

$\to0\quad as~j\to\infty,$
we acquire
\begin{equation*}
	\int_{0}^{t_0}\eta_\varepsilon(t_0-\sigma)(u(t_0),u^{j}(\sigma)){\rm d}\sigma\to\;\int_{0}^{t_0}\eta_\varepsilon(t_0-\sigma)(u(t_0),u(\sigma)){\rm d}\sigma~as~j\to \infty.
\end{equation*}
Similarly, we have
\begin{equation*}
	\int_{0}^{t_0}\eta_\varepsilon(-\sigma)(u_0,u^{j}(\sigma)){\rm d}\sigma\to\;\int_{0}^{t_0}\eta_\varepsilon(-\sigma)(u_0,u(\sigma)){\rm d}\sigma~as~j\to \infty,
\end{equation*}
\begin{equation*}
	\int_{0}^{t_0}\int_{0}^{t_0}\frac{d}{d\tau}[\eta_\varepsilon(\tau-\sigma)](u(\tau),u^{j}(\sigma)){\rm d}\sigma {\rm d}\tau\to\;\int_{0}^{t_0}\int_{0}^{t_0}\frac{d}{d\tau}[\eta_\varepsilon(\tau-\sigma)](u(\tau),u(\sigma)){\rm d}\sigma {\rm d}\tau
\end{equation*}
as $j\to\;\infty$.

Finally, we obtain
\begin{eqnarray}\label{eq3.2}
	& &\int_{0}^{t_0}\eta_\varepsilon(t_0-\sigma)(u(t_0),u(\sigma)){\rm d}\sigma \nonumber\\
	&=&\int_{0}^{t_0}\eta_\varepsilon(-\sigma)(u_0,u(\sigma)){\rm d}\sigma \nonumber\\
	& &-\int_{0}^{t_0}\int_{0}^{t_0}\eta_\varepsilon(\tau-\sigma)\left[((u\cdot\nabla)u(\tau),u(\sigma)) +(|D(u)|^{r-2}D(u),D(u)(\sigma))\right]{\rm d}\sigma {\rm d}\tau\nonumber\\
	& & +\int_{0}^{t_0}\int_{0}^{t_0}\frac{d}{d\tau}[\eta_\varepsilon(\tau-\sigma)](u(\tau),u(\sigma)){\rm d}\sigma {\rm d}\tau
\end{eqnarray}

Sending $\varepsilon$ to zero in \eqref{eq3.2}, and $\varepsilon<t_0$. For the term on the left of \eqref{eq3.2}, we have
\begin{equation*}
	\int_{0}^{t_0}\eta_\varepsilon(t_0-\sigma)(u(t_0),u(\sigma)){\rm d}\sigma
	=\int_{0}^{\varepsilon}\eta_\varepsilon(\tau)(u(t_0),u(t_0-\tau)){\rm d}\tau.
\end{equation*}
Since $u$ is weakly continuous in the sense of $L^{2}$, $W$ is weakly continuous in the sense of $L^{2}$, then
\begin{equation*}
	(u(t_0),u(t_0-\tau))=\|u(t_0)\|^{2}_{L^{2}}+\beta(\tau),
\end{equation*}
where $\beta(\tau)\to\;0$ as $\sigma\to\;0$. Hence, as $\varepsilon\to\;0$ and mollifier is even, we gain
\begin{equation*}
	\int_{0}^{t_0}\eta_\varepsilon(t_0-\sigma)(u(t_0),u(\sigma)){\rm d}\sigma=\int_{0}^{\varepsilon}\eta_\varepsilon(\tau)(\|u(t_0)\|^{2}_{L^{2}}+\beta(\tau)){\rm d}\tau\to\;\frac{1}{2}\|u(t_0)\|^{2}_{L^{2}}
\end{equation*}
\begin{equation*}
	\int_{0}^{t_0}\eta_\varepsilon(-\sigma)(u_0,u(\sigma)){\rm d}\sigma\to\;\frac{1}{2}\|u_0\|^{2}_{L^{2}}
\end{equation*}as $\varepsilon\to 0$.\\

For the last term in \eqref{eq3.2}, since $\eta_\varepsilon$ is even, we have
\begin{eqnarray*}
	& & \int_{0}^{t_0}\int_{0}^{t_0}\frac{d}{d\tau}[\eta_\varepsilon(\tau-\sigma)](u(\tau),u(\sigma)){\rm d}\sigma {\rm d}\tau \\
	&=& -\int_{0}^{t_0}\int_{0}^{t_0}\frac{d}{d\sigma}[\eta_\varepsilon(\tau-\sigma)](u(\tau),u(\sigma)){\rm d}\sigma {\rm d}\tau \\
	&=& -\int_{0}^{t_0}\int_{0}^{t_0}\frac{d}{d\sigma}[\eta_\varepsilon(\sigma-\tau)](u(\tau),u(\sigma)){\rm d}\sigma {\rm d}\tau \\
	&=& -\int_{0}^{t_0}\int_{0}^{t_0}\frac{d}{d\sigma}[\eta_\varepsilon(\sigma-\tau)](u(\sigma),u(\tau)){\rm d}\tau {\rm d}\sigma \\
	&=& -\int_{0}^{t_0}\int_{0}^{t_0}\frac{d}{d\tau}[\eta_\varepsilon(\tau-\sigma)](u(\tau),u(\sigma)){\rm d}\sigma {\rm d}\tau.
\end{eqnarray*}
Thus,
\begin{equation*}
	\int_{0}^{t_0}\int_{0}^{t_0}\frac{d}{d\tau}[\eta_\varepsilon(\tau-\sigma)](u(\tau),u(\sigma)){\rm d}\sigma {\rm d}\tau=0.
\end{equation*}

 The difficulty lies in estimating the nonlinear terms.  Here we present the estimate for the nonlinear terms $|D(u)|^{r-2}D(u)$ and $(u\cdot\nabla)u$.

	Before estimating the  term $\left(|D(u)|^{r-2}D(u),D(u)\right)$, we  first note that  the following fact:
	\begin{equation*}
		\|D(u)\|_{L^r}\leq\frac{1}{2}\left(\|\nabla u\|_{L^r}+\|(\nabla u)^{T}\|_{L^r}\right)=\|\nabla u\|_{L^r},
	\end{equation*}
	
	then we can naturally get
	\begin{eqnarray*}
		& & \left|\int_{0}^{t_0}\int_{0}^{t_0}\eta_\varepsilon(\tau-\sigma)\left( |D(u)|^{r-2}D(u),D(u^{j})-D(u)\right){\rm d}\sigma {\rm d}\tau\right|\\
		&\leq& C(\varepsilon)\left|\int_{0}^{t_0}\int_{0}^{t_0}\||D(u)|^{r-1}(\tau)\|_{L^{\frac{r}{r-1}}}\|D(u^{j}-u)(\sigma)\|_{L^{r}}{\rm d}\sigma {\rm d}\tau\right| ~(\text{H\"{o}lder's inequality})\\
		&=&	C(\varepsilon)\int_{0}^{t_0}\int_{0}^{t_0}\||D(u)|(\tau)\|^{r-1}_{L^{r}}\|D(u^{j}-u)(\sigma)\|_{L^{r}}\;{\rm d}\sigma {\rm d}\tau \\
		&\leq&C(\varepsilon)\int_{0}^{t_0}\|\nabla u(\tau)\|^{r-1}_{L^{r}}{\rm d}\tau\int_{0}^{t_0}\|\nabla(u^{j}-u)(\sigma)\|_{L^{r}}\;{\rm d}\sigma  \\
		&\leq& C(\varepsilon,t_0)\|\nabla u\|_{L^{r}(L^{r})}\|\|\nabla(u^{j}-u)\|_{L^{r}(L^{r})} \quad(\text{ H\"{o}lder's inequality})\\
        &\to&0
	\end{eqnarray*}
	as $j~\to \infty $, where the exponents of  H\"{o}lder's inequality are expressed as
	\begin{equation*}
		1=\frac{1}{\frac{r}{r-1}}+\frac{1}{r}.
	\end{equation*}
	Thus,
	\begin{eqnarray*}
		& & \int_{0}^{t_0}\int_{0}^{t_0}\eta_\varepsilon(\tau-\sigma)\left( |D(u)|^{r-2}D(u),D(u^{j})\right){\rm d}\sigma {\rm d}\tau \\
		&\to& \int_{0}^{t_0}\int_{0}^{t_0}\eta_\varepsilon(\tau-\sigma)\left( |D(u)|^{r-2}D(u),D(u)\right){\rm d}\sigma {\rm d}\tau.
	\end{eqnarray*}
	as$~j~\to \infty $.
	Since
	\begin{eqnarray*}
		& & \left|\int_{0}^{t_0}\int_{0}^{t_0}\eta_\varepsilon(\tau-\sigma)\left( |D(u)|^{r-2}D(u),D(u)(\sigma)\right){\rm d}\sigma {\rm d}\tau-\int_{0}^{t_0}\left( |D(u)|^{r-2}D(u),D(u)\right){\rm d}\tau\right|\\
		&=&\left|\int_{0}^{t_0}\eta_\varepsilon(\tau-\sigma)\left( |D(u)|^{r-2}D(u),D(u_\varepsilon)(\tau)-D(u)(\tau)\right) {\rm d}\tau\right|\\
		&\leq& \left|\int_{0}^{t_0}\||D(u)|^{r-1}(\tau)\|_{L^{\frac{r}{r-1}}}\|D(u_\varepsilon-u)(\tau)\|_{L^{r}} {\rm d}\tau\right|\quad(\text{ H\"{o}lder's inequality})\\
		&=&	\int_{0}^{t_0}\|D(u)(\tau)\|^{r-1}_{L^{r}}\|D(u_\varepsilon-u)\|_{L^{r}}{\rm d}\tau \\
		&=&\int_{0}^{t_0}\|\nabla u(\tau)\|^{r-1}_{L^{r}}\|\nabla(u_\varepsilon-u)\|_{L^{r}}{\rm d}\tau \\
		&\leq& \|\nabla u\|_{L^{r}(L^{r})}\|\|\nabla(u_\varepsilon-u)\|_{L^{r}(L^{r})} \quad(\text{ H\"{o}lder's inequality})\\
        &\to&0
	\end{eqnarray*}
	as$~\varepsilon~\to 0 $, where the exponents of  H\"{o}lder's inequality are expressed as

	\begin{equation*}
		1=\frac{1}{\frac{r}{r-1}}+\frac{1}{r}.
	\end{equation*}
	Thus,
	\begin{equation*}
		\int_{0}^{t_0}\int_{0}^{t_0}\eta_\varepsilon(\tau-\sigma)\left( |D(u)|^{r-2}D(u),D(u)(\sigma)\right){\rm d}\sigma {\rm d}\tau\to\int_{0}^{t_0}\left( |D(u)|^{r-2}D(u),D(u)\right){\rm d}\tau
	\end{equation*}
	as$~\varepsilon~\to 0 $.
	
	According to Lemma \ref{le2.7} we have
	\begin{eqnarray*}
		& &\left|\int_{0}^{t_0}\int_{0}^{t_0}\eta_\varepsilon
		(\tau-\sigma)((u\cdot\nabla)u(\tau),u^{j}(\sigma)){\rm d}\sigma{\rm d}\tau-\int_{0}^{t_0}((u\cdot\nabla)u(\tau),u(\tau)){\rm d}\tau \right|\\
		&=&\left|\int_{0}^{t_0}((u\cdot\nabla)u(\tau),u^{j}_\varepsilon(\tau)-u(\tau)){\rm d}\tau \right| \\
		&=& \left|\int_{0}^{t_0}((u\cdot\nabla)(u^{j}_\varepsilon(\tau)-u(\tau)),u(\tau)){\rm d}\tau \right| \\
		&\leq& \||u|^{\frac{2}{\alpha+\theta}-2}u\|^{\alpha+\theta}_{L^{\frac{r(5r-6)(\alpha +\theta)}{5r^{2}-(11+2\theta)r+6}}\mathcal{M}^{3}(\dot{H}^{\frac{\theta}{\alpha+\theta}}\to L^{\frac{r(\alpha+\theta)}{r-1}})}\|\nabla u\|^{\frac{2r\theta}{5r-6}}_{L^{r}(L^{r})}\\
        & &\|u\|^{\frac{\alpha(5r-6)+3\theta(r-2)}{5r-6}}_{L^{\infty}(L^{2})}\|\nabla (u^{j}_\varepsilon-u)\|_{L^{r}(L^{r})}\\
        &\to& 0
	\end{eqnarray*}
	 as $j\to \infty$ and $\varepsilon\to 0$, \\ where the exponents of   H\"{o}lder's inequality are expressed as
	\begin{equation*}
		1=\frac{1}{\frac{r(5r-6)}{5r^{2}-(11+2\theta)r+6}}+\frac{1}{\frac{5r-6}{2\theta}}+\frac{1}{r}.
	\end{equation*}
	Therefore, we obtain
	\begin{equation*}
		\int_{0}^{t_0}\int_{0}^{t_0}\eta_\varepsilon
		(\tau-\sigma)((u\cdot\nabla)u(\tau),u^{j}(\sigma)){\rm d}\sigma{\rm d}\tau\to\int_{0}^{t_0}((u\cdot\nabla)u(\tau),u(\tau)){\rm d}\tau .
	\end{equation*}
	Taking into account $(u\cdot\nabla u,u)=0$, we have,
	\begin{equation*}
		\int_{0}^{t_0}\int_{0}^{t_0}\eta_\varepsilon
		(\tau-\sigma)((u\cdot\nabla)u(\tau),u^{j}(\sigma)){\rm d}\sigma{\rm d}\tau\to0
	\end{equation*}
	
	as $j\to \infty$ and $~\varepsilon~\to 0$.
	
Concluding the results and sending $\varepsilon\to~0$ in \eqref{eq3.2}, we attain
\begin{equation*}
	\|u(t_0)\|^{2}_{L^{2}}+2\int_{0}^{t_0}\|D(u)(\tau)\|^{r}_{L^{r}}{\rm d}\tau=\|u_0\|^{2}_{L^{2}}
\end{equation*}
i.e. \eqref{eq1.2}.
\end{proof}

\section{Proofs of Theorem \ref{th1.2} and Theorem \ref{th1.3}}\label{Sec4}
Under the hypotheses of Theorem \ref{th1.2} or Theorem \ref{th1.3}, in order to estimate the advection term in \eqref{eq1.1}, we need the next  two lemmas respectively:

\begin{lemma}\label{le2.8}
			For $u\in L^{\infty}(0,T;L^{2}_{\sigma}(\mathbb{R}^{3}))\cap L^{r}(0,T;W^{1,r}_{\sigma}(\mathbb{R}^{3}))$   and $u$ satisfies 	\begin{equation*}
			|u|^{\frac{10\theta^{2}-23\theta+7}{3+8\theta-5\theta^{2}}}u\in L^{\frac{(5r-6)(-5\theta^{2}+8\theta+3)}{(5r-6)(-\theta+4)+2(5\theta^{2}-8\theta-3)}}\mathcal{M}^{3}(\dot{H}^{1}\to L^{\frac{-5\theta^{2}+8\theta+3}{3-2\theta}}),
		\end{equation*}
		\begin{equation*}
			|\nabla u|^{\frac{10-5\theta}{7-3\theta}} \in L^{\frac{(\alpha+1)(7-3\theta)(5r-6)}{2(3-2\theta)(\alpha+1)(5r-6)-2(7-3\theta)}}\mathcal{M}^{3}(\dot{H}^{\frac{1}{\alpha+1}}\to L^{1}),
		\end{equation*}
		where $ \alpha\geq0,~r\geq2,~0\leq\theta<\frac{1}{2}$, we have the following estimate:
		\begin{eqnarray*}
			& &\left|\int_{0}^{t_0}\left( (u \cdot \nabla) u,u\right) {\rm d}\tau\right|\\
			&\lesssim&\left\|\left|\nabla u\right|^{\frac{10-5\theta}{7-3\theta}} \right\|^{\frac{7-3\theta}{10-5\theta}}_{L^{\frac{(\alpha+1)(7-3\theta)(5r-6)}{2(3-2\theta)(\alpha+1)(5r-6)-2(7-3\theta)}}\mathcal{M}^{3}(\dot{H}^{\frac{1}{\alpha+1}}\to L^{1})}\|\nabla u\|^{\frac{2r(7-3\theta)}{(\alpha+1)(5r-6)(10-5\theta)}}_{L^{r}(L^{r})}\\
& &\|u\|^{\frac{(5\alpha r+3r-6\alpha-6)(7-3\theta)}{(\alpha+1)(5r-6)(10-5\theta)}}_{L^{\infty}(L^{2})}\left\| |u|^{\frac{10\theta^{2}-23\theta+7}{3+8\theta-5\theta^{2}}}u\right\|^{\frac{(-5\theta^{2}+8\theta+3)(13-7\theta)^{2}}{(10-5\theta)(4-\theta)^{2}}}_{{L^{\frac{(5r-6)(s-\gamma)}{(5r-6)(2\gamma-1)-2(s-\gamma)}}}\mathcal{M}^{3}(\dot{H}^{1}\to L^{\frac{-5\theta^{2}+8\theta+3}{3-2\theta}})}\\
& &\|\nabla u\|^{\frac{2r}{5r-6}\cdot\frac{(-5\theta^{2}+8\theta+3)(13-7\theta)^{2}}{(10-5\theta)(4-\theta)^{2}}}_{L^{r}(L^{r})}\|u\|^{\frac{3r-6}{5r-6}\cdot\frac{(-5\theta^{2}+8\theta+3)(13-7\theta)^{2}}{(10-5\theta)(4-\theta)^{2}}}_{L^{\infty}(L^{2})}
		\end{eqnarray*}
		
	\end{lemma}
	\begin{lemma}\label{le2.9}
			For $u\in L^{\infty}(0,T;L^{2}_{\sigma}(\mathbb{R}^{3}))\cap L^{r}(0,T;W^{1,r}_{\sigma}(\mathbb{R}^{3}))$ and $u$ satisfies 	
		\begin{equation*}
			\nabla u \in L^{p}\mathcal{M}^{3}(\dot{W}^{\frac{1+\theta}{1+\alpha},\frac{6-3\theta}{3-\theta^{2}}}\to L^{\frac{12-6\theta}{9-5\theta}}),
		\end{equation*}
		\begin{equation*}
			u \in L^{q}(L^{\frac{12-6\theta}{3-\theta}}),~ \frac{1}{p}+\frac{1}{q}+\frac{1}{r}=1,~2\leq p<\infty,
		\end{equation*}
		where $ \alpha\geq0,~r\ge2,~-1\leq\theta<0$, we have
		\begin{eqnarray*}
			& &\left|\int_{0}^{t_0}\left((u \cdot \nabla) u,u\right ) {\rm d}\tau\right|\\
			&\lesssim_{t_0}&\left\|u \right\|_{L^{q}(L^{\frac{12-6\theta}{3-\theta}})}\|\nabla u\|_{L^{p}\mathcal{M}^{3}(\dot{W}^{\frac{1+\theta}{1+\alpha},\frac{6-3\theta}{3-\theta^{2}}}\to L^{\frac{12-6\theta}{9-5\theta}})}\\
& &(\|\nabla u\|_{L^{r}(L^{r})}+1)\|u\|^{1-\frac{r(-\theta+4-3\alpha\theta+2\alpha\theta^{2})}{(5r-6)(1+\alpha)(2-\theta)}}_{L^{\infty}(L^{2})}.
		\end{eqnarray*}
		
	\end{lemma}
For the sake of compactness, the proofs of Lemma \ref{le2.8} and Lemma \ref{le2.9} are detailed in the appendix. \\

\begin{proof}[Sketch of proof of Theorem \ref{th1.2}]	
	Since the estimates of  the linear terms in \eqref{eq3.1} are the same as those in Theorem \ref{th1.1},  we omit them. And we only give the estimate about the nonlinear term,
	For Theorem \ref{th1.2}, On the basis of Lemma \ref{le2.8} we acquire
	\begin{eqnarray*}
		& &\left|\int_{0}^{t_0}\left( u(\tau) \cdot \nabla (u^{j}_\varepsilon(\tau)-u(\tau)),u(\tau)\right ) {\rm d}\tau\right|\\
		&\lesssim&\left\|\left|\nabla(u^{j}_\varepsilon-u)\right|^{\frac{1}{\gamma}} \right\|^{\gamma}_{L^{\frac{\gamma(\alpha+1)(5r-6)}{2(1-\gamma)(\alpha+1)(5r-6)-2\gamma}}\mathcal{M}^{3}(\dot{H}^{\frac{1}{\alpha+1}}\to L^{1})}\\
& &\|\nabla u\|^{\frac{2r\gamma}{(\alpha+1)(5r-6)}}_{L^{r}(L^{r})}\|u\|^{\frac{\alpha\gamma(\beta+1)-\gamma}{\alpha(\beta+1)}}_{L^{\infty}(L^{2})}~\|u\|^{2-\gamma}_{L^{\frac{2-\gamma}{2\gamma-1}}(L^{\frac{2-\gamma}{1-\gamma}})},
	\end{eqnarray*}where the exponents satisfy the following relation,
	\begin{equation*}
		\begin{cases}
			\frac{1}{\alpha+1}-\frac{3}{2}=(1-\frac{3}{r})\frac{2r}{(\alpha+1)(5r-6)} + (-\frac{3}{2})(1-\frac{2r}{(\alpha+1)(5r-6)}),\\
			1=\frac{1}{\frac{(\alpha+1)(5r-6)}{2(1-\gamma)(\alpha+1)(5r-6)-2\gamma}}+\frac{1}{\frac{(\alpha+1)(5r-6)}{2\gamma}}+\frac{1}{\frac{1}{2\gamma-1}},
		\end{cases}
	\end{equation*}
	where $\gamma=\frac{7-3\theta}{10-5\theta}\in [\frac{21}{30},\frac{22}{30})$, and $r\geq2$ guarantees
	
	\begin{equation*}
		\begin{cases}
			1-\frac{2r}{(\alpha+1)(5r-6)}\geq0,\\
		2(1-\gamma)(\alpha+1)(5r-6)-2\gamma\geq8-10\gamma>0.
		\end{cases}
	\end{equation*}

	 For $\|u\|_{L^{\frac{2-\gamma}{2\gamma-1}},L^{\frac{2-\gamma}{1-\gamma}}}$, taking $s=\frac{15\gamma-10}{5\gamma-3}$, we have
	\begin{eqnarray*}
		& &\|u\|^{\frac{2\gamma-1}{2-\gamma}}_{L^{\frac{2-\gamma}{2\gamma-1}},L^{\frac{2-\gamma}{1-\gamma}}}\\
		&\lesssim&\left\| |u|^{\frac{2-\gamma}{s-\gamma}-2}u\right\|^{\frac{s-\gamma}{2\gamma-1}}_{{L^{\frac{(5r-6)(s-\gamma)}{(5r-6)(2\gamma-1)-2(s-\gamma)}}}\mathcal{M}^{3}(\dot{H}^{1}\to L^{\frac{s-\gamma}{1-\gamma}})}\|\nabla u\|^{\frac{2r}{5r-6}\cdot\frac{s-\gamma}{2\gamma-1}}_{L^{r}(L^{r})}\|u\|^{\frac{3r-6}{5r-6}\cdot\frac{s-\gamma}{2\gamma-1}}_{L^{\infty}(L^{2})},
	\end{eqnarray*}where the exponents satisfy the following relation,
	\begin{equation*}
		\begin{cases}
			1-\frac{3}{2}=(1-\frac{3}{r})\frac{2r}{5r-6} + (-\frac{3}{2})(1-\frac{2r}{5r-6}),\\
			1=\frac{1}{\frac{(5r-6)(2\gamma-1)}{(5r-6)(2\gamma-1)-2(s-\gamma)}}+\frac{1}{\frac{(5r-6)(2\gamma-1)}{2(s-\gamma)}}.
		\end{cases}
	\end{equation*}
	and $r\geq2$ guarantees $(5r-6)(2\gamma-1)-2(s-\gamma)\geq10\gamma-4-2s=\frac{50(\gamma-\frac{4}{5})^{2}}{5\gamma-3}>0$.
	
\end{proof}

\begin{proof}[Sketch of proof of Theorem \ref{th1.3}]
	In the light of  Lemma \ref{le2.9}, we gain
	\begin{eqnarray*}
		& &\left|\int_{0}^{t_0}\left(u(\tau) \cdot \nabla u(\tau),u^{j}_\varepsilon(\tau)-u(\tau)\right ) {\rm d}\tau\right|\\
		&\lesssim&\int_{0}^{t_0}\left\|u^{j}_\varepsilon(\tau)-u (\tau)\right\|_{L^{\frac{12-6\theta}{3-\theta}}}\|\nabla u\|_{\mathcal{M}^{3}(\dot{W}^{\frac{1+\theta}{1+\alpha},\frac{6-3\theta}{3-\theta^{2}}}\to L^{\frac{12-6\theta}{9-5\theta}})}\|\nabla u\|^{\frac{r(-\theta+4-3\alpha\theta+2\alpha\theta^{2})}{(5r-6)(1+\alpha)(2-\theta)}}_{L^{r}}\\
		& &\|u\|^{1-\frac{r(-\theta+4-3\alpha\theta+2\alpha\theta^{2})}{(5r-6)(1+\alpha)(2-\theta)}}_{L^{2}}{\rm d}\tau\quad(*)\\
		&\lesssim_{t_0}&\left\|u^{j}_\varepsilon-u \right\|_{L^{q}(L^{\frac{12-6\theta}{3-\theta}})}\|\nabla u\|_{L^{p}\mathcal{M}^{3}(\dot{W}^{\frac{1+\theta}{1+\alpha},\frac{6-3\theta}{3-\theta^{2}}}\to L^{\frac{12-6\theta}{9-5\theta}})}(\|\nabla u\|_{L^{r}(L^{r})}+1)\\
		& &\|u\|^{1-\frac{r(-\theta+4-3\alpha\theta+2\alpha\theta^{2})}{(5r-6)(1+\alpha)(2-\theta)}}_{L^{\infty}(L^{2})},
	\end{eqnarray*}where the exponents satisfy the following relation,
	\begin{equation*}
		\begin{cases}
			\frac{1+\theta}{1+\alpha}-\frac{3}{\frac{6-3\theta}{3-\theta^{2}}} =  (1-\frac{3}{r})\frac{r(-\theta+4-3\alpha\theta+2\alpha\theta^{2})}{(5r-6)(1+\alpha)(2-\theta)}+  (-\frac{3}{2})({1-\frac{r(-\theta+4-3\alpha\theta+2\alpha\theta^{2})}{(5r-6)(1+\alpha)(2-\theta)}}),\\
			1=\frac{1}{r}+\frac{1}{q}+\frac{1}{p},~2\leq p<\infty.
		\end{cases}
	\end{equation*}
		and $r\geq2$ guarantees $0<\frac{r(-\theta+4-3\alpha\theta+2\alpha\theta^{2})}{(5r-6)(1+\alpha)(2-\theta)}<1$.
\end{proof}

Actually, we can replace the exponent of H\"{o}lder's inequality $	1=\frac{1}{r}+\frac{1}{q}+\frac{1}{p}$ with $	1=\frac{1}{\frac{(5r-6)(1+\alpha)(2-\theta)}{-\theta+4-3\alpha\theta+2\alpha\theta^{2}}}+\frac{1}{q'}+\frac{1}{p'}$ at the asterisk(*), when $r\geq2$, we have $\frac{(5r-6)(1+\alpha)(2-\theta)}{-\theta+4-3\alpha\theta+2\alpha\theta^{2}}>r$, $p'$ has a wider range of values than $p$, and Theorem \ref{th1.3} still holds true.

\section{Appendix}\label{Sec5}
For the sake of completeness and reader convenience, we present the detailed proofs of Lemma \ref{le2.7}, Lemma \ref{le2.8},  and Lemma \ref{le2.9} respectively below.

\begin{proof}[Proof of Lemma \ref{le2.7}]
	\begin{eqnarray*}
		& &\left|\int_{0}^{t_0}((u\cdot\nabla)u,u){\rm d}\sigma \right| \\
		&\leq&\int_{0}^{t_0}\|u^{2}\|_{L^{\frac{r}{r-1}}}\|\nabla u\|_{L^{r}} {\rm d}\tau \quad(\text{ H\"{o}lder's inequality})\\
		&=&  \int_{0}^{t_0}\| u^{\frac{2}{\alpha+\theta}-2}u\cdot u\|^{\alpha+\theta}_{L^{\frac{r(\alpha+\theta)}{r-1}}}\|\nabla u\|_{L^{r}}{\rm d}\tau \\
		&\leq&  \int_{0}^{t_0}\||u|^{\frac{2}{\alpha+\theta}-2}u\|^{\alpha+\theta}_{\mathcal{M}^{3}(\dot{H}^{\frac{\theta}{\alpha+\theta}}\to L^{\frac{r(\alpha+\theta)}{r-1}})}\|u\|^{\alpha+\theta}_{\dot{H}^{\frac{\theta}{\alpha+\theta}}}\|\nabla u\|_{L^{r}}{\rm d}\tau \\
		&\lesssim& \int_{0}^{t_0}\||u|^{\frac{2}{\alpha+\theta}-2}u\|^{\alpha+\theta}_{\mathcal{M}^{3}(\dot{H}^{\frac{\theta}{\alpha+\theta}}\to L^{\frac{r(\alpha+\theta)}{r-1}})}\|\nabla u\|^{\frac{2r\theta}{5r-6}}_{L^{r}}\| u\|^{\frac{\alpha(5r-6)+3\theta(r-2)}{5r-6}}_{L^{2}}\|\nabla u\|_{L^{r}}{\rm d}\tau\\
		& &(\text{Gagliardo-Nirenberg inequality})\\
		&\lesssim& \left\|\||u|^{\frac{2}{\alpha+\theta}-2}u\|^{\alpha+\theta}_{\mathcal{M}^{3}(\dot{H}^{\frac{\theta}{\alpha+\theta}}\to L^{\frac{r(\alpha+\theta)}{r-1}})}\right\|_{L^{\frac{r(5r-6)}{5r^{2}-(11+2\theta)r+6}}}\left\|\|\nabla u\|^{\frac{2r\theta}{5r-6}}_{L^{r}}\right\|_{L^{\frac{5r-6}{2\theta}}}\\
& &~~\| u\|^{\frac{\alpha(5r-6)+3\theta(r-2)}{5r-6}}_{L^{\infty}(L^{2})}\|\nabla u\|_{L^{r}(L^{r})} ~~(\text{ H\"{o}lder's inequality})\\
		&\leq& \||u|^{\frac{2}{\alpha+\theta}-2}u\|^{\alpha+\theta}_{L^{\frac{r(5r-6)(\alpha +\theta)}{5r^{2}-(11+2\theta)r+6}}\mathcal{M}^{3}(\dot{H}^{\frac{\theta}{\alpha+\theta}}\to L^{\frac{r(\alpha+\theta)}{r-1}})}\|\nabla u\|^{\frac{2r\theta}{5r-6}}_{L^{r}(L^{r})}\\
& &~~\|u\|^{\frac{\alpha(5r-6)+3\theta(r-2)}{5r-6}}_{L^{\infty}(L^{2})}\|\nabla u\|_{L^{r}(L^{r})},
	\end{eqnarray*}
where the   aforementioned exponents satisfy the following relation
	\begin{equation*}
		1=\frac{1}{\frac{r(5r-6)}{5r^{2}-(11+2\theta)r+6}}+\frac{1}{\frac{5r-6}{2\theta}}+\frac{1}{r}.
	\end{equation*}
\end{proof}

\begin{proof}[Proof of Lemma \ref{le2.8}]
	\begin{eqnarray*}
		& &\left|\int_{0}^{t_0}\left( (u \cdot \nabla) u,u\right) {\rm d}\tau\right|\\
		&\leq&\int_{0}^{t_0}\left\||u|^{\gamma}~\nabla u \right\|_{L^{\frac{1}{\gamma}}}\||u|^{2-\gamma}\|_{L^{\frac{1}{1-\gamma}}}{\rm d}\tau\quad(\text{ H\"{o}lder's inequality})\\
		&=&\int_{0}^{t_0}\left\|u~\left|\nabla u\right|^{\frac{1}{\gamma}} \right\|^{\gamma}_{L^{1}} \|u\|^{2-\gamma}_{L^{\frac{2-\gamma}{1-\gamma}}}{\rm d}\tau\\
		&\leq&\int_{0}^{t_0} \left\|\left|\nabla u\right|^{\frac{1}{\gamma}}  \right\|^{\gamma}_{\mathcal{M}^{3}(\dot{H}^{\frac{1}{\alpha+1}}\to L^{1})}\|u\|^{\gamma}_{\dot{H}^{\frac{1}{\alpha+1}}}\|u\|^{2-\gamma}_{L^{\frac{2-\gamma}{1-\gamma}}}{\rm d}\tau\\
		&\lesssim&\int_{0}^{t_0} \left\|\left|\nabla u\right|^{\frac{1}{\gamma}}  \right\|^{\gamma}_{\mathcal{M}^{3}(\dot{H}^{\frac{1}{\alpha+1}}\to L^{1})}\|\nabla u\|^{\frac{2r\gamma}{(\alpha+1)(5r-6)}}_{L^{r}}\|u\|^{\gamma-\frac{2r\gamma}{(\alpha+1)(5r-6)}}_{L^{2}}\|u\|^{2-\gamma}_{L^{\frac{2-\gamma}{1-\gamma}}}{\rm d}\tau\\
		& &\quad(\text{Gagliardo-Nirenberg inequality})\\
		&\leq&\left\| \left\|\left|\nabla u\right|^{\frac{1}{\gamma}}  \right\|^{\gamma}_{\mathcal{M}^{3}(\dot{H}^{\frac{1}{\alpha+1}}\to L^{1})}  \right\|     _{L^{\frac{(\alpha+1)(5r-6)}{2(1-\gamma)(\alpha+1)(5r-6)-2\gamma}}}\left\|\|\nabla u\|^{\frac{2r\gamma}{(\alpha+1)(5r-6)}}_{L^{r}}\right\|_{L^{\frac{(\alpha+1)(5r-6)}{2\gamma}}}\\
		& &\|u\|^{\gamma-\frac{2r\gamma}{(\alpha+1)(5r-6)}}_{L^{\infty}(L^{2})}~\left\|\|u\|^{2-\gamma}_{L^{\frac{2-\gamma}{1-\gamma}}}\right\|_{L^{\frac{1}{2\gamma-1}}}\quad(\text{ H\"{o}lder's inequality})\\
		&=&\left\|\left|\nabla u\right|^{\frac{1}{\gamma}} \right\|^{\gamma}_{L^{\frac{\gamma(\alpha+1)(5r-6)}{2(1-\gamma)(\alpha+1)(5r-6)-2\gamma}}\mathcal{M}^{3}(\dot{H}^{\frac{1}{\alpha+1}}\to L^{1})}\|\nabla u\|^{\frac{2r\gamma}{(\alpha+1)(5r-6)}}_{L^{r}(L^{r})}\\
& &~~\|u\|^{\gamma-\frac{2r\gamma}{(\alpha+1)(5r-6)}}_{L^{\infty}(L^{2})}~\|u\|^{2-\gamma}_{L^{\frac{2-\gamma}{2\gamma-1}}(L^{\frac{2-\gamma}{1-\gamma}})},
	\end{eqnarray*} where the exponents satisfy the following relation,
	\begin{equation*}
		\begin{cases}
			\frac{1}{\alpha+1}-\frac{3}{2}=(1-\frac{3}{r})\frac{2r}{(\alpha+1)(5r-6)} + (-\frac{3}{2})(1-\frac{2r}{(\alpha+1)(5r-6)}),\\
			1=\frac{1}{\frac{(\alpha+1)(5r-6)}{2(1-\gamma)(\alpha+1)(5r-6)-2\gamma}}+\frac{1}{\frac{(\alpha+1)(5r-6)}{2\gamma}}+\frac{1}{\frac{1}{2\gamma-1}},
		\end{cases}
	\end{equation*}
	and $\gamma=\frac{7-3\theta}{10-5\theta}\in [\frac{21}{30},\frac{22}{30})$, and $r\geq2$ guarantees
	\begin{equation*}
		\begin{cases}
			1-\frac{2r}{(\alpha+1)(5r-6)}\geq0,\\
			2(1-\gamma)(\alpha+1)(5r-6)-2\gamma\geq8-10\gamma>0.
		\end{cases}
	\end{equation*}
	
	For $\|u\|_{L^{\frac{2-\gamma}{2\gamma-1}},L^{\frac{2-\gamma}{1-\gamma}}}$, taking $s=\frac{15\gamma-10}{5\gamma-3}$, we have
	\begin{eqnarray*}
		& &\|u\|^{\frac{2\gamma-1}{2-\gamma}}_{L^{\frac{2-\gamma}{2\gamma-1}},L^{\frac{2-\gamma}{1-\gamma}}}\\
		&=&\int_{0}^{t_0}\|u\|^{{\frac{2-\gamma}{2\gamma-1}}}_{L^{\frac{2-\gamma}{1-\gamma}}}{\rm d}\tau\\
		&\leq&\int_{0}^{t_0}\left\| |u|^{\frac{2-\gamma}{s-\gamma}-2}u\cdot u\right\|^{\frac{s-\gamma}{2\gamma-1}}_{L^{\frac{s-\gamma}{1-\gamma}}}{\rm d}\tau\\
		&\leq&\int_{0}^{t_0}\left\| |u|^{\frac{2-\gamma}{s-\gamma}-2}u\right\|^{\frac{s-\gamma}{2\gamma-1}}_{\mathcal{M}^{3}(\dot{H}^{1}\to L^{\frac{s-\gamma}{1-\gamma}})}\|u\|^{\frac{s-\gamma}{2\gamma-1}}_{\dot{H}^{1}}{\rm d}\tau\\
		&\leq&\int_{0}^{t_0}\left\| |u|^{\frac{2-\gamma}{s-\gamma}-2}u\right\|^{\frac{s-\gamma}{2\gamma-1}}_{\mathcal{M}^{3}(\dot{H}^{1}\to L^{\frac{s-\gamma}{1-\gamma}})}\|\nabla u\|^{\frac{2r}{5r-6}\cdot\frac{s-\gamma}{2\gamma-1}}_{L^{r}}\|u\|^{\frac{3r-6}{5r-6}\cdot\frac{s-\gamma}{2\gamma-1}}_{L^{2}}{\rm d}\tau\\
		& &\quad(\text{Gagliardo-Nirenberg inequality})\\
		&\leq&\left\|\left\| |u|^{\frac{2-\gamma}{s-\gamma}-2}u\right\|^{\frac{s-\gamma}{2\gamma-1}}_{\mathcal{M}^{3}(\dot{H}^{1}\to L^{\frac{s-\gamma}{1-\gamma}})}\right\|_{L^{\frac{(5r-6)(2\gamma-1)}{(5r-6)(2\gamma-1)-2(s-\gamma)}}}\left\|\|\nabla u\|^{\frac{2r}{5r-6}\cdot\frac{s-\gamma}{2\gamma-1}}_{L^{r}} \right\|_{L^{\frac{(5r-6)(2\gamma-1)}{2(s-\gamma)}}}\\
& &~~\|u\|^{\frac{3r-6}{5r-6}\cdot\frac{s-\gamma}{2\gamma-1}}_{L^{\infty}(L^{2})} \quad(\text{ H\"{o}lder's inequality})\\
		&=&\left\| |u|^{\frac{2-\gamma}{s-\gamma}-2}u\right\|^{\frac{s-\gamma}{2\gamma-1}}_{{L^{\frac{(5r-6)(s-\gamma)}{(5r-6)(2\gamma-1)-2(s-\gamma)}}}\mathcal{M}^{3}(\dot{H}^{1}\to L^{\frac{s-\gamma}{1-\gamma}})}\|\nabla u\|^{\frac{2r}{5r-6}\cdot\frac{s-\gamma}{2\gamma-1}}_{L^{r}(L^{r})}\|u\|^{\frac{3r-6}{5r-6}\cdot\frac{s-\gamma}{2\gamma-1}}_{L^{\infty}(L^{2})},
	\end{eqnarray*}where the exponents satisfy the following relation,
	\begin{equation*}
		\begin{cases}
			1-\frac{3}{2}=(1-\frac{3}{r})\frac{2r}{5r-6} + (-\frac{3}{2})(1-\frac{2r}{5r-6}),\\
			1=\frac{1}{\frac{(5r-6)(2\gamma-1)}{(5r-6)(2\gamma-1)-2(s-\gamma)}}+\frac{1}{\frac{(5r-6)(2\gamma-1)}{2(s-\gamma)}}.
		\end{cases}
	\end{equation*}
	and $r\geq2$ guarantees $$(5r-6)(2\gamma-1)-2(s-\gamma)\geq10\gamma-4-2s=\frac{50(\gamma-\frac{4}{5})^{2}}{5\gamma-3}>0.$$
\end{proof}

\begin{proof}[Proof of Lemma \ref{le2.9}]
	\begin{eqnarray*}
		& &\left|\int_{0}^{t_0}\left((u \cdot \nabla) u,u\right ) {\rm d}\tau\right|\\
		&\leq&\int_{0}^{t_0}\left\|u \right\|_{L^{\frac{12-6\theta}{3-\theta}}}\|u\cdot \nabla u\|_{L^{\frac{12-6\theta}{9-5\theta}}}{\rm d}\tau\quad(\text{ H\"{o}lder's inequality})\\
		&\leq&\int_{0}^{t_0}\left\|u \right\|_{L^{\frac{12-6\theta}{3-\theta}}}\|\nabla u\|_{\mathcal{M}^{3}(\dot{W}^{\frac{1+\theta}{1+\alpha},\frac{6-3\theta}{3-\theta^{2}}}\to L^{\frac{12-6\theta}{9-5\theta}})}\|u\|_{\dot{W}^{\frac{1+\theta}{1+\alpha},\frac{6-3\theta}{3-\theta^{2}}}}{\rm d}\tau\\
		&\lesssim& \int_{0}^{t_0}\left\|u \right\|_{L^{\frac{12-6\theta}{3-\theta}}}\|\nabla u\|_{\mathcal{M}^{3}(\dot{W}^{\frac{1+\theta}{1+\alpha},\frac{6-3\theta}{3-\theta^{2}}}\to L^{\frac{12-6\theta}{9-5\theta}})}\|\nabla u\|^{\frac{r(-\theta+4-3\alpha\theta+2\alpha\theta^{2})}{(5r-6)(1+\alpha)(2-\theta)}}_{L^{r}}\\
		& &\|u\|^{1-\frac{r(-\theta+4-3\alpha\theta+2\alpha\theta^{2})}{(5r-6)(1+\alpha)(2-\theta)}}_{L^{2}}{\rm d}\tau\quad(\text{Gagliardo-Nirenberg inequality})\\
		&=&\int_{0}^{t_0}\left\|u \right\|_{L^{\frac{12-6\theta}{3-\theta}}}\|\nabla u\|_{\mathcal{M}^{3}(\dot{W}^{\frac{1+\theta}{1+\alpha},\frac{6-3\theta}{3-\theta^{2}}}\to L^{\frac{12-6\theta}{9-5\theta}})}\|\nabla u\|^{\frac{r(-\theta+4-3\alpha\theta+2\alpha\theta^{2})}{(5r-6)(1+\alpha)(2-\theta)}}_{L^{r}}\\
		& &\|u\|^{1-\frac{r(-\theta+4-3\alpha\theta+2\alpha\theta^{2})}{(5r-6)(1+\alpha)(2-\theta)}}_{L^{2}}{\rm d}\tau\quad(*)\\
		&\lesssim& C\int_{0}^{t_0}\left\|u \right\|_{L^{\frac{12-6\theta}{3-\theta}}}\|\nabla u\|_{\mathcal{M}^{3}(\dot{W}^{\frac{1+\theta}{1+\alpha},\frac{6-3\theta}{3-\theta^{2}}}\to L^{\frac{12-6\theta}{9-5\theta}})}\\
& &~~(\|\nabla u\|_{L^{r}}+1)\|u\|^{1-\frac{r(-\theta+4-3\alpha\theta+2\alpha\theta^{2})}{(5r-6)(1+\alpha)(2-\theta)}}_{L^{2}}{\rm d}\tau\\
		&\lesssim&\left\|u \right\|_{L^{q}(L^{\frac{12-6\theta}{3-\theta}})}\|\nabla u\|_{L^{p}\mathcal{M}^{3}(\dot{W}^{\frac{1+\theta}{1+\alpha},\frac{6-3\theta}{3-\theta^{2}}}\to L^{\frac{12-6\theta}{9-5\theta}})}\\
& &~~\left\|\|\nabla u\|_{L^{r}}+1\right\|_{L^{r}}\|u\|^{1-\frac{r(-\theta+4-3\alpha\theta+2\alpha\theta^{2})}{(5r-6)(1+\alpha)(2-\theta)}}_{L^{\infty}(L^{2})} \quad(\text{ H\"{o}lder's inequality})\\
		&\lesssim_{t_0}&\left\|u \right\|_{L^{q}(L^{\frac{12-6\theta}{3-\theta}})}\|\nabla u\|_{L^{p}\mathcal{M}^{3}(\dot{W}^{\frac{1+\theta}{1+\alpha},\frac{6-3\theta}{3-\theta^{2}}}\to L^{\frac{12-6\theta}{9-5\theta}})}\\
& &~~\left(\|\nabla u\|_{L^{r}(L^{r})}+1\right)\|u\|^{1-\frac{r(-\theta+4-3\alpha\theta+2\alpha\theta^{2})}{(5r-6)(1+\alpha)(2-\theta)}}_{L^{\infty}(L^{2})},
	\end{eqnarray*}where the exponents satisfy the following relation,
	\begin{equation*}
		\begin{cases}
			\frac{1+\theta}{1+\alpha}-\frac{3}{\frac{6-3\theta}{3-\theta^{2}}} =  (1-\frac{3}{r})\frac{r(-\theta+4-3\alpha\theta+2\alpha\theta^{2})}{(5r-6)(1+\alpha)(2-\theta)}+  (-\frac{3}{2})({1-\frac{r(-\theta+4-3\alpha\theta+2\alpha\theta^{2})}{(5r-6)(1+\alpha)(2-\theta)}}),\\
			1=\frac{1}{r}+\frac{1}{q}+\frac{1}{p},~2\leq p<\infty.
		\end{cases}
	\end{equation*}
	and $r\geq2$ guarantees $0<\frac{r(-\theta+4-3\alpha\theta+2\alpha\theta^{2})}{(5r-6)(1+\alpha)(2-\theta)}<1$.
\end{proof}

\section*{Data availability statement}
No new data were created or analysed in this study.

\section*{Acknowledgements}
This work was supported by National Natural Science Foundation of China (Grant No. 12271470).

	

\begin{thebibliography}{100}
		\setlength{\itemsep}{-2mm}
		
		\bibitem{BY19} Beir\~{a}o da Veiga H,   Yang J.  On the energy equality for solutions to Newtonian and non-Newtonian fluids. Nonlinear Anal., 2019,  {\bf 185}: 388--402
		
		\bibitem{BC20}  Berselli L C,  Chiodaroli E.  On the energy equality for the 3D Navier-Stokes equations. Nonlinear Anal., 2020, {\bf 192}:  Art 111704
		
	\bibitem{BuKP2019} Bul\'{\i}\v{c}ek M,  Kaplick\'{y} P,  Pra\v{z}\'{a}k D.  Uniqueness and regularity of flows of non-Newtonian fluids with critical power-law growth. Math. Models Methods Appl. Sci., 2019, {\bf  29}  no. 6, 1207--1225
	\bibitem{BuMM2023}  Bul\'{\i}\v{c}ek M,   M\'{a}lek J,  Maringov\'{a} E.   On unsteady internal flows of incompressible fluids characterized by implicit constitutive equations in the bulk and on the boundary. J. Math. Fluid Mech., 2023, {\bf  25}, no. 3, Paper No. 72, 29 pp.
		\bibitem{CZ23}  Chen Q L,  Zhang Q.  Energy equality of the weak solutions to Navier-Stokes equations in the multiplier spaces. Appl. Math. Lett., 2023, {\bf 146}  No. 108814, 6 pp.
		\bibitem{CheCoFrSh08}  Cheskidov A,  Constantin P,  Friedlander S,  Shvydkoy R.  Energy conservation and Onsager's conjecture for the Euler equations. Nonlinearity, 2008, {\bf  21}: 1233--1252
		
        \bibitem{ChesLuo20}  Cheskidov A,  Luo X.  Energy equality for the Navier-Stokes equations in weak-in-time Onsager spaces. Nonlinearity, 2020, {\bf  33}: 1388--1403

\bibitem{CrFG2026} Crispo F,  Feola A. Pia Di,  Grisanti C  R. Estimates of a possible gap related to the energy equality for a class of non-Newtonian fluids.  Nonlinear Anal. TMA, 2026, {\bf 268}: No. 114080

		\bibitem{ELC98}  Evans L C. {Partial differential equations. Second edition}. Graduate Studies in Mathematics, 19. American Mathematical Society, Providence, RI, 2010. xxii+749 pp.
		
        \bibitem{FW25}Feng Y, Wang W.  Energy Equalities of the Inhomogeneous Navier-Stokes Equations, MHD Equations and Hall-MHD Equations. Bull. Malays. Math. Sci. Soc., 2025, {\bf  48}  Art 166	
        \bibitem{FW26} Feng Y, Wang W. Energy equality for the Navier-Stokes-Maxwell equations. J. Evol. Equ., 2026, {\bf 26} Art 51


        \bibitem{Galdi08}G. P. Galdi,  Mathematical problems in classical and non-Newtonian fluid mechanics. In Hemodynamical flows. Modeling, Analysis and Simulation, Oberwolfach Seminars, 37, 121--273. Birkh\"{a}user, Basel, 2008

		
		\bibitem{GGP19}  Galdi G P.   On the energy equality for distributional solutions to Navier-Stokes equations. Proc. Amer. Math. Soc., 2019, {\bf 147}:   785--792
		
		\bibitem{LJL60}   Lions J L.  Sur la r\'{e}gularit\'{e} et l'unicit\'{e} des solutions turbulentes des \'{e}quations de Navier Stokes(French). Rend. Semin. Mat.
		Univ. Padova 30 (1960) 16--23.
	
		\bibitem{MK84}  Masuda K.   Weak solutions of Navier-Stokes equations. Tohoku Math. J., 1984, {\bf 36}:  623--646
		\bibitem{MS09}  Maz'ya N G,  Shaposhnikova T O.  Theory of Sobolev multipliers. With applications to differential and integral operators. in: Grundlehren der mathematischen Wissenschaften [Fundamental Principles of Mathematical Sciences], Vol.337. Springer-Verlag, Berlin, (2009).
		

       \bibitem{Nirenberg11}  Nirenberg L.  On elliptic partial differential equations. Principio di minimo e sue applicazioni alle equazioni funzionali (Berlin, Heidelberg: Springer, 2011), pp. 1--48
	
		\bibitem{Onsager1949} Onsager L.  Statistical hydrodynamics, Nuovo Cimento (9) 6, (Supplemento, 2 (Convegno Internazionale di Meccanica Statistica)), (1949), 279--287
		
		
		\bibitem{SJ63}  Serrin J.   The initial-value problem for the Navier-Stokes equations. In: Langer, R.E. (ed.) Nonlinear Problems. University of Wisconsin Press, Madison (1963) 69--98
		
		\bibitem{SM74}  Shinbrot M.   The energy equation for the Navier-Stokes system. SIAM J. Math. Anal., 1974, {\bf 5}:   948--954
	    \bibitem{SinBa23} Sin C,     Baranovskii E S.   Regularity criterion for 3D generalized Newtonian fluids in BMO. J. Differential Equations, 2023, {\bf 377}:  859--872
        \bibitem{WangMH22}  Wang Y,  Mei X,   Huang Y.    Energy equality of the 3D Navier-Stokes equations and generalized Newtonian equations, J. Math. Fluid Mech., 2022, {\bf 24}(3), Paper No. 65, 10 pp.
		
       \bibitem{WJ07}  Wolf J.  Existence of weak solutions to the equations of non-stationary motion of non-Newtonian fluids with shear rate dependent viscosity.  J. Math. Fluid Mech., 2007, {\bf 9}(1): 104--138
		
		\bibitem{Wu24}  Wu F.  A note on energy equality for the fractional Navier-Stokes equations,  Proceedings of the Royal Society of Edinburgh, 2024, {\bf 154}: 201--208
		
		
		
		\bibitem{YJ19}  Yang J.  The energy equality for weak solutions to the equations of non-Newtonian fluids. Appl. Math. Lett., 2019, {\bf 88}:  216--221
		
		\bibitem{ZZ19}  Zhang Z.   Remarks on the energy equality for the non-Newtonian fluids. J. Math. Anal. Appl. 2019, {\bf480}(2), Art 123443, 9 pp.	

		\end{thebibliography}
	\end{document}